\pdfoutput=1
\documentclass[oneside
,12pt
]
{amsart}
\usepackage[verbose=true, a4paper, scale={.7,.8}]{geometry}
\usepackage{graphicx}
\usepackage{mathrsfs}
\theoremstyle{plain}
\newtheorem{theorem}{Theorem}

\newtheorem{corollary}{Corollary}
\newtheorem{proposition}{Proposition}
\theoremstyle{definition}
\newtheorem*{definition}{Definition}
\newtheorem*{remark}{Remark}
\newcommand{\D}{\ensuremath{\mathbf{D}}}
\title[Shifting CEI processes at random]{Shifting processes with cyclically exchangeable increments at random}
\author{Lo\"\i c Chaumont}
\address{LAREMA, D\'epartement de Math\'ematiques, Universit\'e d'Angers,
 2, Bd Lavoisier-49045, Angers Cedex 01. France. 
}
\email{loic.chaumont@univ-angers.fr}
\author{Ger\'onimo Uribe Bravo}
\address{Instituto de  Matem\'aticas, Universidad Nacional Aut\'onoma de M\'exico, \'Area de la Investigaci\'on Cient\'ifica, Ciudad Universitaria 04510, Ciudad de M\'exico, M\'exico}

\email{geronimo@matem.unam.mx}
\newcommand{\eps}{\ensuremath{ \varepsilon}}
\newenvironment{esn}{\begin{equation*}}{\end{equation*}}
\newcommand{\imf}[2]{\ensuremath{#1\!\paren{#2}}}
\newcommand{\paren}[1]{\ensuremath{\left( #1\right) }}
\newcommand{\fun}[3]{\ensuremath{#1:#2\to #3}}
\newcommand{\re}{\ensuremath{\mathbb{R}}}
\newcommand{\set}[1]{\ensuremath{\left\{ #1\right\} }}
\newcommand{\indi}[1]{\si_{#1}}
\newcommand{\si}{{\ensuremath{\bf{1}}}}
\newcommand{\cadlag}{c\`adl\`ag}
\newcommand{\sa}{\ensuremath{\sigma}\nbd field}
\newcommand{\nbd}{\nobreakdash -}
\newcommand{\sag}[1]{\sigma\!\paren{#1}}
\newcommand{\floor}[1]{\ensuremath{\lfloor #1\rfloor}}
\newcommand{\defin}[1]{{\bf #1}}
\newcommand{\p}{\ensuremath{ \sip  } }
\newcommand{\sip}{\bb{P}}
\newcommand{\bb}[1]{\mathbb{#1}}
\newcommand{\proba}[1]{\ensuremath{\sip\! \left( #1 \right)}}
\newcommand{\esp}[1]{\ensuremath{\se\! \left( #1 \right)}}
\newcommand{\se}{\ensuremath{\bb{E}}}
\newcommand{\sko}{D}
\newcommand{\abs}[1]{\hspace{.25mm}\left|#1\right|\hspace{.25mm}}
\DeclareMathOperator{\id}{Id} %
\newcommand{\bra}[1]{\ensuremath{\left[ #1\right] }}
\newcommand{\oo}{\ensuremath{ \Omega  } }
\newcommand{\espc}[2]{\ensuremath{\imf{\se}{\cond{#1}{#2}}}}
\newcommand{\cond}[2]{\left.\vphantom{#2}#1\ \right| #2}
\newcommand{\F}{\ssa}
\newcommand{\ssa}{\ensuremath{\mathscr{F}}}
\DeclareMathOperator{\sgn}{sgn} %
\usepackage[dvipsnames]{xcolor}
\usepackage[colorlinks,citecolor=Plum,urlcolor=Periwinkle,linkcolor=DarkOrchid]{hyperref}

\newcommand{\fp}[1]{\ensuremath{\left\{ #1 \right\}}}
\newcommand{\ed}{\stackrel{(d)}{=}}
\begin{document}
\begin{abstract}
We propose a path transformation which applied to a 
cyclically exchangeable increment process conditions its minimum to belong to
a given interval. 

This path transformation is then applied to processes with
start and end at $0$. It is seen that, under simple conditions, 
the weak limit as $\eps\to 0$ of the process conditioned on remaining above
$-\eps$ exists  and has the law of the Vervaat transformation of the
process. 

We examine the consequences of this path transformation on processes with exchangeable increments, L\'evy bridges, and the Brownian bridge. 
\end{abstract}
\maketitle

\section{Motivation: Weak convergence of conditioned Brownian bridge and the Vervaat transformation}
\label{motivationSection}
Excursion theory for Markov processes has proved to be an useful tool since its inception in \cite{MR0402949} (although some ideas date back to \cite{MR0000919}). 
This is true both in theoretical and applied investigations (see for example \cite{MAFI:MAFI349}, \cite{multiplicativeCoalescentAldous},  \cite{MR2295611}, \cite{MR2603058}, \cite{MR2603061}, \cite{MR2603062}, \cite{MR3134857}). 
Especially fruitful has been the application of excursion theory in the case of Brownian motion and other L\'evy processes, where it lies at the foundation of the so called fluctuation theory aimed at studying their extremes (cf. \cite{MR1406564}, \cite{MR2250061}, \cite{MR2320889}). 

Brownian motion is undoubtedly one of the most tractable L\'evy processes. 
It is therefore not a surprise that excursion theory takes a very explicit form for this process. 
In particular, we have the following interpretation of the normalized excursions above $0$ of Brownian motion as Brownian motion conditioned to start at $0$, end at $0$ at a given time $t$ (here we consider $t=1$), and remain positive throughout $(0,t)$. 
We recall that the Brownian bridge is a version of Brownian motion conditioned to start at $0$ and end at $0$ at time $1$. 
\begin{theorem}[\cite{MR0436353} and \cite{MR515820}]
\label{DIMVTheorem}
Let $X$ be a Brownian bridge from $0$ to $0$ of length $1$. 
Then, the law of $X$ conditioned to remain above $-\eps$ converges weakly as $\eps\to 0$ toward law of the normalized Brownian excursion.  
Furthermore, the weak limit can also be constructed as follows: if $\rho$ is the unique instant at which $X$ attains its minimum, then the weak limit  has the same law as
\begin{esn}
\imf{\theta_\rho}{ X}_t= 
\begin{cases}
m+X_{\rho+t}& 0\leq t\leq 1-\rho\\
m+X_{t-(1-\rho)}&1-\rho\leq t\leq 1
\end{cases}.
\end{esn}
\end{theorem}
The first assertion was proved in \cite{MR0436353} by showing the convergence of the finite-dimensional distributions and then by establishing tightness through explicit computations for Brownian motion. 
The weak convergence of the finite-dimensional distributions (fdd) is a simple consequence of having explicit expressions for the transition densities for Brownian motion and Brownian motion killed upon reaching zero, 
as can be seen in Lemma 5.2 of \cite{MR0436353}. 
They are then seen to coincide with the fdd of the three-dimensional Bessel bridge from $0$ to $0$ of length $1$. 
Other approaches to Theorem \ref{DIMVTheorem} are found in \cite{MR704566} (in terms of computations with generators) and \cite{MR3160578} (for more general bridges of L\'evy processes). 
One consequence of this work is to show that tightness follows very easily from a path transformation which applied to a 
cyclically exchangeable increment process conditions its minimum to belong to
a given interval. 
The second assertion of Theorem \ref{DIMVTheorem} was shown in \cite{MR515820}; a second consequence of our work is to show that this is a much more general fact. 
We now illustrate the technique in the case of Theorem \ref{DIMVTheorem}. 

The fractional part 
of a real number $t$ will be denoted $\fp{t}$ 
. 
We define the family of transformations $\paren{\theta_u, u\in[0,1]}$, acting on any continuous path $\fun{f}{[0,1]}{\re}$ with $\imf{f}{0}=\imf{f}{1}=0$  by means of
\begin{align}
\label{def3}
\imf{\theta_u f}{t}
&=\imf{f}{\fp{t+u}}-\imf{f}{u}
.
\end{align}

The transformation $\theta_u$ consists in
inverting the paths $\{f(t),\,0\leq t\leq u\}$ and
$\{f(t),\,u\leq t\leq1\}$
. 
See Figure \ref{shiftFigure} for an example.  
\begin{figure}
\includegraphics[width=\textwidth]{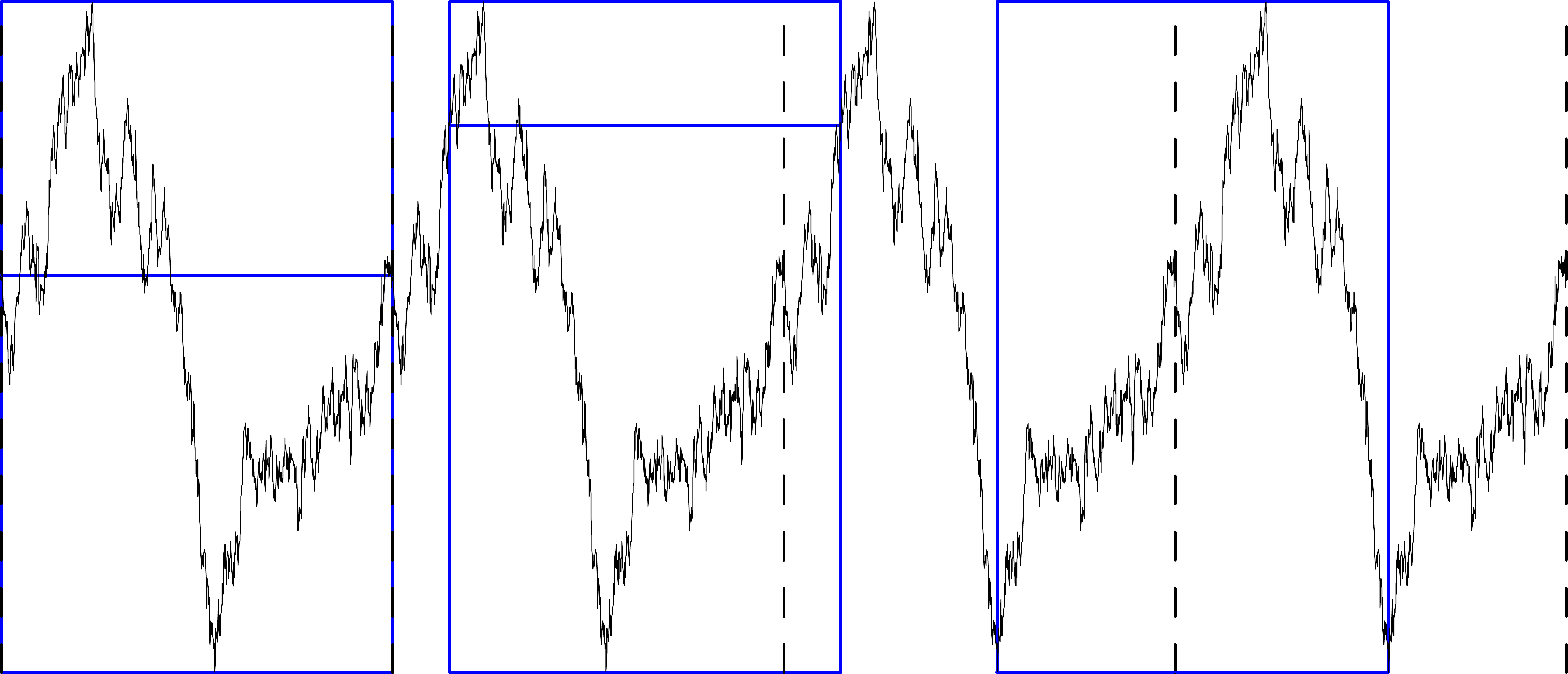}
\caption{Repeated trajectory of a Brownian bridge. The first frame shows the original trajectory. The second shows its shift at $u=.14634$. The third frame shows the shift at the location of the unique minimum, illustrating the Vervaat transformation.}
\label{shiftFigure}
\end{figure}

\begin{proof}[Proof of Theorem \ref{DIMVTheorem}]
Recall that the minimum $\underline X$ of a Brownian bridge is achieved at a unique place, say $\rho$, 
as can be proved simply using the Gaussian character of $X$. 
Note that the minimum is transformed as follows under a shift: $\underline X\circ\theta_t= \underline X-X_t$. 
We now introduce a random distribution which will let us choose a uniform point on the nonempty open set $\set{t\in [0,1]: X_t-\underline X<\eps}$: 
\begin{esn}
F^\eps_t=\frac{\int_0^t  \indi{X_s-\underline X<\eps}\, ds}{\int_0^1 \indi{X_s-\underline X<\eps}\, ds}. 
\end{esn}Note that $F^\eps_0=0$, $F^\eps_1=1$ and that $F^\eps$ is non-decreasing and continuous. 
Let $U$ be a uniform random variable independent of $X$, and set\begin{esn}
\tau_t=\inf\set{s\geq 0: F^\eps_s>t}\quad \text{and}\quad \eta=\tau_{U}. 
\end{esn}
We now perform the random shift $X\circ \theta_\eta$;
the forthcoming Theorem \ref{mainTheorem} (applied with $I=(-\eps,0]$) tells us that this process has the law of $X$ conditioned on $\underline X\geq -\eps$. 
Since $\rho$ is the unique minimim, it follows that\begin{esn}
\set{\rho}=\bigcap_{\eps>0}\set{t\in [0,1]: X_t-\underline X<\eps}, 
\end{esn}so that $\eta\to \rho$ as $\eps\to 0$. 

Now, note that for any continuous $\fun{f}{[0,1]}{\re}$ such that $\imf{f}{0}=\imf{f}{1}$, if $v\to u$ then $\| \theta_{v}f- \theta_u f\|\to 0$. 
Indeed, such an $f$ might be thought of as continuous on the unit circle, and the transformation $\theta_u$ is obtained by applying $f$ to the rotation of $t$ by $u$, which is a continuous operation.
We conclude that $\theta_{\eta} X\to \theta_\rho X$ almost surely and hence that the law of $X$ conditioned on $\underline X\geq -\eps$ converges weakly to the law of $\theta_\rho X$.
\end{proof}

\section{Conditioning the minimum of a process with cyclically exchangeable increments}
\label{MainTheoremSection}
We now turn to our main theorem in the context of cyclically exchangeable increment processes. 

We use the canonical setup: let $\D$ stand for the Skorohod space of \cadlag\ functions $\fun{f}{[0,1]}{\re}$ on which the canonical process $X=\paren{X_t,t\in [0,1]}$ is defined. 
Recall that $\fun{X_t}{\D}{\re}$ is given by\begin{esn}
\imf{X_t}{f}=\imf{f}{t}.
\end{esn}Then, $\D$ is equipped with the \sa\ $\sag{X_t,t\in [0,1]}$. 
We extend the transformation $\theta_u$ defined in \eqref{def3} by setting\begin{equation}
\label{generalShift}
\imf{\theta_u f}{t}=\imf{f}{\fp{t+u}}-\imf{f}{u}+\imf{f}{\floor{t+u}}. 
\end{equation}
The transformation $\theta_u$ consists in
inverting the paths $\{f(t),\,0\leq t\leq u\}$ and
$\{f(t),\,u\leq t\leq1\}$
 in such a way that the new path
$\theta_u(f)$ has the same values as $f$ at
times 0 and 1, i.e. $\theta_u f(0)=f(0)$
and $\theta_uf(1)=f(1)$. 
We call $\theta_u$ the \defin{shift} at time $u$ of $X$ over the interval
[0,1].\\
Note that we will always use the transformation $\theta_u$ with $\imf{f}{0}=0$. 

\begin{definition}[CEI process]
A \cadlag\ stochastic process has \defin{cyclically exchangeable increments (CEI)} if its law satisfies the following identities in law: 
\begin{esn}
\theta_u X\ed X \text{  for every $u\in [0,1]$. }
\end{esn}
\end{definition}

The overall minimum $\underline X$, which can be defined now as a functional on Skorohod space, is given by\begin{esn}
\underline X=\inf_{0\leq t\leq 1} X_t.
\end{esn}

Intuitively, to condition $X$ on having a minimum on a given interval $I\subset (-\infty,0]$, we choose $t$ uniformly on the set in which $\underline X\circ\theta_t \in I$ by using the occupation times process\begin{esn}
A_t^{I}=\int_0^t \indi{\underline X\circ\theta_s \in I}\, ds.
\end{esn}

Here is the main result. 
It provides a way to construct CEI processes conditioned on their overall minimum.
\begin{theorem}
\label{mainTheorem}
Let $(X,\p)$ be any non trivial CEI process  such that $X_0=0$, $X_1\ge0$ and $\proba{\underline X\in I}>0$. 
Let $U$ be an independent random time which is uniformly distributed  over $[0,1]$ and define:
\begin{equation}
\label{randomShiftForMainTheorem}
\nu=\inf\{t:A_t^{I}=UA_1^{I}\}.
\end{equation}
Conditionally on $A_1^{I}>0$, the process $\theta_\nu(X)$ is independent of
$\nu$ and has the same law as $X$ conditionally on 
$\underline{X}\in I$.
Moreover the time $\nu$ is uniformly distributed over $[0,1]$.

Conversely, if $Y$ has the law of $X$ conditioned on $\underline X\in I$ and $U$ is uniform and independent of $Y$ then $\imf{\theta_U}{Y}$ has the same law as $X$ conditioned on $A^I_1>0$. 
\end{theorem}

\begin{remark}
When $X_1=0$, the set $\set{A^I_1>0}$ can be written in terms of the amplitude $A=\overline X-\underline X$ (where $\overline X=\sup_{t\in [0,1]} X_t$) as $\set{A\geq -\inf I}$. 
\end{remark}

\begin{proof}[Proof of Theorem \ref{mainTheorem}]
We first note that the law of $X\circ \theta_U$  conditionally on $\underline X\circ\theta_U\in I$ is equal to the law of $X$ conditionally on $\underline X\in I$. Indeed, using the CEI property:
\begin{align*}
\esp{ \imf{f}{U} F\circ\theta_U \indi{\underline X\circ\theta_U\in I}}
&=\int_0^1\imf{f}{u}\esp{F\circ\theta_u\indi{\underline X\circ\theta_u\in I}} \, du
\\&=\esp{F\indi{\underline X\in I}}\int_0^1\imf{f}{u} \, du.
\end{align*}Additionally, we conclude that the random variable $U$ is uniform on $(0,1)$ and independent of $X\circ\theta_U$ conditionally on $\underline X\circ \theta_U\in I$.

%

Write $U$ in the following way:
\begin{equation}
\label{546}
U=\inf\left\{t:A^{I}_t=\frac{A_U^{I}}{A_1^{I}}A_1^{I}\right\}\,.
\end{equation}
Then it suffices to prove that conditionally on $\underline X\circ\theta_U\in I$, 
the random variable $A_U^{I}/A_1^{I}$ is uniformly distributed over  [0,1] and independent of $X$.
Indeed from the conditional independence and (\ref{546}), we deduce that conditionally on $\underline X\circ\theta_U\in I$, the law of $(\theta_U(X),U)$ is the same as that of $(\imf{\theta_\nu}{X},\nu)$. 

Let $F$ be any positive, measurable functional defined on $\sko$ and $f$ be any positive Borel function. From
the change of variable $s={A_t^{I}}/{A_1^{I}}$, we obtain
\begin{align*}
&\esp{F(X) f({A_U^{I}}/{A_1^{I}})\imf{\indi{I}}{\underline X\circ\theta_U}}
\\
&=\esp{\int_0^1f({A_t^{I}}/{A_1^{I}})F(X)\imf{\indi{I}}{\underline X\circ \theta_t}\,dt}\\
&=\esp{\int_0^1f({A_t^{I}}/{A_1^{I}})F(X)\,dA_t^{I}}\\
&=\esp{F(X)A_1^{I} }\int_0^1f(t)\,dt\,,
\end{align*}
which proves the conditional independence mentioned above.

The converse assertion is immediate using the independence of $\theta_\nu X$ and $\nu$ and the fact that the latter is uniform. 
\end{proof}

We will now apply Theorem \ref{MainTheoremSection} to particular situations to get diverse generalizations of Theorem \ref{DIMVTheorem}.

\section{Exchangeable increment processes and the Vervaat transformation}
\label{EIProcessSection}
In Section \ref{MainTheoremSection} we shifted paths at random using $\theta_\eta$ to condition a given CEI process to have a minimum in a given interval $I$. 
When $I=(-\eps,0]$ and $X_1=0$, and under a simple technical condition, we now see that the limiting transformation of $\theta_\eta$ as $\eps\to 0$ is the Vervaat transformation. 
Hence, we obtain an extension of Theorem \ref{DIMVTheorem}. 

\begin{corollary}
\label{VervaatCorollary}
Let $(X,\p)$ be any non trivial CEI process  such that $X_0=0=X_1$. 
Assume that there exists a unique $\rho\in (0,1)$ such that $X_\rho=\underline X$ and that $X_{\rho-}=X_\rho$. 
Then, the law of $X$ conditioned to remain above $-\eps$ converges weakly in the Skorohod $J_1$ topology as $\eps\to 0$. 
Furthermore, the weak limit is the law of $\theta_\rho X$. 
\end{corollary}
Note that we assume that the infimum of the process $X$ is achieved at $\rho$. 
Actually if the infimum is only achieved as a limit (from the left) at $\rho$ and $X_{\rho-}<X_\rho$ then the transformation $\theta_\eta$ converges, as $\eps\to 0$ pointwise to a process $\theta_\rho$ which satisfies $\imf{\theta_\rho}{0}=0$, $\imf{\theta_{\rho}}{0+}=X_{\rho}-X_{\rho-}$. 
Hence, convergence cannot take place on Skorohod space. 
A similar fact happens when $X_{\rho-}>X_\rho$. 
After the proof, we shall examine an example of applicability of Corollary \ref{VervaatCorollary} to exchangeable increment processes. 
\begin{proof}
We use the notation of Theorem \ref{mainTheorem}. 
In particular, $\eta=\imf{\eta}{\eps}$ is a uniform  point on the set\begin{esn}
\set{t:X_t-\underline X_t<\eps}.
 \end{esn}
As in the proof of Theorem \ref{DIMVTheorem}, the uniqueness of the minimum implies that $\eta\to \rho$ as $\eps\to 0$. 
Since $X$ is continuous at $\rho$, by assumption, for any $\gamma>0$ we can find $\delta>0$ such that $\abs{X_\rho-X_s}<\gamma$ if $s\in [\rho-\delta,\rho+\delta]$. 
On $[0,\rho-\delta]$ and $[\rho+\delta,1]$, we use the \cadlag\ character of $X$ to construct partitions $0=t^1_0<\cdots<t^1_{n_1}=\rho-\delta$ and $\rho+\delta=t^2_0<\cdots<t^2_{n_2}=1$ such that\begin{esn}
\abs{X_s-X_t}<\gamma \quad\text{if}\quad s,t\in [t^i_{j-1},t^i_j)\text{ for }j\leq n_i. 
\end{esn}We use these partitions to construct the piecewise linear increasing homeomorphism $\fun{\lambda}{[0,1]}{[0,1]}$ which satisfies $\|\theta_\eta\circ \lambda-\theta_\rho\|_{[0,1]}\leq \gamma$. 
Indeed, construct $\lambda$ which scales the interval $[0,t^2_1-\eta]$ to $[0,t^2_1-\rho]$, shifts every interval $[t^2_{i-1}-\eta,t^2_{i}-\eta]$ to $[t^2_{i-1}-\rho,t^2_{i}-\rho]$ for $i\leq n_2$, also shifts $[1-\eta+t^1_{i-1},1-\eta+t^1_{i}]$ to $[1-\rho+t^1_{i-1},1-\rho+t^1_{i}]$, and finally scales $[1-\eta+(\rho-\delta),1]$ to $[1-\delta,1]$. 
Note that by choosing $\eta$ close enough to $\rho$, which amounts to choosing $\eps$ small enough, we can make $\| \lambda-\id\|_{[0,1]}\leq \gamma$. 
Hence, $\theta_\eta\to \theta_\rho$ in the Skorohod $J_1$ topology as $\eps\to 0$. 
\end{proof}

Our main example of the applicability of Corollary \ref{VervaatCorollary} is to exchangeable increment processes. 
\begin{definition}
A \cadlag\ stochastic process has \defin{exchangeable increments (EI)} if its law satisfies that for every $n\geq 1$, the random variables \begin{esn}
X_{k/n}-X_{(k-1)/n}, 1\leq k\leq n
\end{esn}are exchangeable. 
\end{definition}
According to \cite{MR0394842}, an EI process has the following canonical representation:
\begin{esn}
X_t=\alpha t+\sigma b_t+\sum_{i} \beta_i \bra{\indi{U_i\leq t}-t}
\end{esn}where\begin{enumerate}
\item $\alpha$, $\sigma$ and $\beta_i, i\geq 1$ are (possibly dependent) random variables such that $\sum_i \beta_i^2<\infty$ almost surely.
\item $b$ is a Brownian bridge
\item $\paren{U_i,i\geq 1}$ are iid uniform random variables on $(0,1)$.
\end{enumerate}Furthermore, the three groups of random variables are independent and the sum defining $X_t$ converges uniformly in $L_2$ in the sense that\begin{esn}
\lim_{m\to\infty}\sup_{n\geq m}\esp{\sup_{t\in [0,1]}\bra{\sum_{i=m+1}^n \beta_i^2\bra{\indi{U_i\leq t}-t}^2}}=0. 
\end{esn}The above representation is called the canonical representation of $X$ and the triple $\paren{\alpha,\beta,\sigma}$ are its canonical parameters. 

Our main example follows from the following result:
\begin{proposition}
\label{EIExampleProposition}
 Let $X$ be an EI process with canonical parameters $\paren{\alpha,\beta,\sigma}$. 
On the set
\begin{esn}
\set{\sum_i \beta_i^2 \abs{\log\abs{\beta_i}}^c<\infty\text{ for some $c>1$} \text{ or } \sigma\neq 0},
\end{esn}$X$ reaches its minimum continuously at a unique $\rho\in (0,1)$. 
\end{proposition}

We need some preliminaries to prove Proposition \ref{EIExampleProposition}. 
First, a criterion to decide whether $X$ has infinite or finite variation in the case there is no Brownian component.
\begin{proposition}
\label{variationProposition}
Let $X$ be an EI process with canonical parameters $\paren{\alpha,\beta,0}$. Then, the sets
\begin{esn}
\set{X\text{ has infinite variation on any subinterval of } [0,1]}
\end{esn}and\begin{esn}
\set{\sum_i\abs{\beta_i}=\infty}
\end{esn}coincide almost surely. 
If $\sum_i\abs{\beta_i}<\infty$ then $X_t/t$ has a limit as $t\to 0$. 
\end{proposition}
It is known that for finite-variation L\'evy processes, $X_t/t$ converges to the drift of $X$ as $t\to 0$ as shown in \cite{MR0183022}. 
\begin{proof}
We work conditionally on $\paren{\beta_i}$; assume then that this sequence is deterministic. 
If $\sum_{i}\abs{\beta_i}<\infty$, we can define the following two increasing processes
\begin{esn}
X^p_t=\alpha^+ t+\sum_{i:\beta_i>0} \beta_i\indi{U_i\leq t}
\quad\text{and}\quad
X^n_t=\alpha^{-} t+ \sum_{i:\beta_i<0} -\beta_i\indi{U_i\leq t}
\end{esn}and note that $X=X^p-X^n$. Hence $X$ has bounded variation on $[0,1]$ almost surely. 

On the other hand, if $\sum_{i}\abs{\beta_i}=\infty$ we first assert that the set\begin{esn}
A_{k,n}=\set{\sum_i \abs{\beta_i}\indi{k/n\leq U_i\leq (k+1)/n}=\infty}
\end{esn}has probability $1$  for any $n\geq 1$ and any $k\in\set{0,\ldots, n}$. 
Note that for fixed $n$, $\cup_{0\leq k\leq n-1} A_{k,n}=\oo$. 
Also, $\proba{A_{k_1,n}}=\proba{A_{k_2,n}}$ since the $U_i$ are uniform. 
Finally, note that $A_{k,n}$ belongs to the tail \sa\ of the sequence of random variables $\paren{U_i}$. 
Hence, $\proba{A_{k_1,n}}=1$ by the Kolmogorov 0-1 law. 
Since\begin{esn}
\sum_{i}\abs{\beta_i }\indi{a\leq U_i\leq b}=\sum_{t: \Delta X_t\neq 0}\abs{\Delta X_t}\indi{a\leq t\leq b}
\end{esn}and the sum of jumps of a \cadlag\ function is a lower bound for the variation, we see that $X$ has infinite variation on any subinterval of $[0,1]$. 

Recall that $X_t\in L_2$ (since we assumed that the canonical parameters are constant). 
Using the EI property, it is easy to see that\begin{esn}
\espc{X_s}{X_t,t\geq s}=\frac{s}{t}  X_t.
\end{esn} 
Hence the process $M=\paren{M_t,t\in [0,1)}$ given by $M_t=X_{1-t}/(1-t)$ is a martingale. 
If  $\sum_{i}\abs{\beta_i}<\infty$ 
then\begin{esn}
X_t=\alpha t+\sum_i \beta_i\indi{U_i\leq t}-t\sum_{i}\beta_i
\end{esn}so that $\esp{\abs{M_t}}\leq \abs{\alpha} + 2\sum_i \abs{\beta_i}$. 
Hence, $M$ is bounded in $L_1$ as $t\to 1$ and so it converges almost surely. 
\end{proof}

Secondly, we give a version of a result originally found in  \cite{MR0242261} for L\'evy processes. 
\begin{proposition}
\label{rogozinTypeProposition}
The set\begin{esn}
\set{\sum_i\abs{\beta_i}=\infty, \sum_i \beta_i^2 \abs{\log \abs{\beta_i}}^c<\infty \text{ for some }c>1 \text{ or }\sigma\neq 0} 
\end{esn}is almost surely contained in\begin{esn}
\set{\limsup_{t\to\infty} \frac{X_t}{t}=\infty\text{ and }\liminf_{t\to\infty} \frac{X_t}{t}=-\infty}. 
\end{esn}
\end{proposition}
\begin{proof}
By conditioning on the canonical parameters, we will assume they are constant. 

If $\sigma\neq 0$, let $\fun{f}{[0,1]}{\re}$ be a function such that $\sqrt{t}=\imf{o}{\imf{f}{t}}$ and 
$\imf{f}{t}=\imf{o}{\paren{t\log\log1/t}^{1/2}}$ 
as $t\to 0$. 
Then, since the law of the Brownian bridge is equivalent to the law of $B$ on any interval $[0,t]$ for $t<1$, the law of the iterated logarithm implies that  
$\limsup_{t\to 0}b_t/\imf{f}{t}=\infty$ and $\liminf_{t\to 0}b_t/\imf{f}{t}=-\infty$. 
On the other hand, if $Y=X-\sigma b$, then $Y$ is an EI process with canonical parameters $\paren{\alpha,\beta,0}$ independent of $b$. 
Note that $\esp{Y_t^2}\sim t\sum_{i}\beta_i^2$ as $t\to 0$ to see that $Y_t/\imf{f}{t}\to 0$ in $L_2$ as $t\to 0$. If $t_n$ is a (random and $b$-measurable) sequence decreasing to zero such that $b_{t_n}/\imf{f}{t_n}$ goes to $\infty$, we can use the independence of $Y$ and $b$ to construct a subsequence $s_n$ converging to zero such that $b_{s_n}/\imf{f}{s_n}\to \infty$ and $Y_{s_n}/\imf{f}{s_n}\to 0$. We conclude that $X_{t_n}/t_n\to\infty$ and so $\limsup_{t\to 0} X_t/t=\infty$. 
The same argument applies for the lower limit. 

Let us now assume that $\sigma=0$. 
If $\sum_i\abs{\beta_i}=\infty$ then $X$ has infinite variation on any subinterval of $[0,1]$. 
If furthermore  $\sum_i \beta_i^2 \abs{\log \abs{\beta_i}}^c<\infty$ then Theorem 1.1 of \cite{MR0402901} allows us to write $X$ as $Y+Z$ where $Z$ is of finite variation process with exchangeable increments and $Y$ is a L\'evy process. 
Since $Y$ has infinite variation, then $\liminf_{t\to 0} Y_t/t=-\infty$ and $\limsup_{t\to 0}Y_t/t =\infty$ thanks to \cite{MR0242261}. 
Finally, since $\lim_{t\to 0} Z_t/t$ exists in $\re$ by Proposition \ref{variationProposition} since $Z$ is a finite variation EI process, then $\liminf_{t\to 0}X_t/t=-\infty$ and $\limsup_{t\to 0}X_t/t=\infty$. 
\end{proof}
\begin{proof}[Proof of Proposition \ref{EIExampleProposition}]
Since $\liminf_{t\to 0}X_t/t=-\infty$, we see that $\rho>0$. Using the exchangeability of the increments, we conclude from Proposition \ref{rogozinTypeProposition} that at any deterministic $t\geq 0$ we have that $\limsup_{h\to 0} (X_{t+h}-X_t)/h=\infty$ and $\liminf_{h\to 0} (X_{t+h}-X_t)/h=-\infty$ almost surely. 
If we write $X_t=\beta_i\bra{\indi{U_i\leq t}-t}+X'_t$ and use the independence between $U_i$ and $X'$, we conclude that at any jump time $U_i$ we have:
$\limsup_{h\to 0} (X_{U_i+h}-X_{U_i})/h=\infty$ and $\liminf_{h\to 0} (X_{U_i+h}-X_t)/h=-\infty$ almost surely. 
We conclude from this that $X$ cannot jump into its minimum. 
By applying the preceeding argument to $\paren{X_1-X_{\paren{1-t}-},t\in [0,1]}$, which is also EI with the same canonical parameters, we see that $\rho<1$ and that $X$ cannot jump into its minimum either. 
\end{proof}

In contrast to the case of EI processes where we have only stated sufficient conditions for the achievement of the minimum, necessary and sufficient conditions are known for L\'evy processes. 
Indeed, Theorem 3.1 of \cite{MR0433606} tells us that if $X$ is a L\'evy process such that neither $X$ nor $-X$ is a subordinator, 
then $X$ achieves its minimum continuously if and only if $0$ is regular for $(0,\infty)$ and $(-\infty,0)$. 
This happens always when $X$ has infinite variation. 
In the finite-variation case, 
regularity of $0$ for $(0,\infty)$ can be established through Rogozin's criterion: $0$ is regular for $(-\infty,0)$ if and only if $\int_{0+} \proba{X_t<0}/t\, dt=\infty$. 
A criterion in terms of the characteristic triple of the L\'evy process is available in \cite{MR1465812}. 
We will therefore assume\begin{description}
\item[H1] $0$ is regular for $(-\infty,0)$ and $(0,\infty)$.
\end{description}

We now proceed then to give a statement of a Vervaat type transformation for L\'evy processes, although actually we will use their bridges in order to force them to end at zero. 
L\'evy bridges were first constructed in \cite{MR628873} (using the convergence criteria for processes with exchangeable increments of \cite{MR0394842}) and then in \cite{MR2789508} (via Markovian considerations) under the following hypothesis: 
\begin{description}
\item[H2] For any $t > 0$, $\int\abs{\esp{e^{iu X_t}}}\, du<\infty$. 
\end{description}
Under \defin{H2}, the law of $X_t$ is absolutely continuous with a continuous and bounded density $f_t$. 
Hence, $X$ admits transition densities $p_t$ given by $\imf{p_t}{x,y}=\imf{f_t}{y-x}$. 
If we additionally assume \defin{H1} then the transition densities are everywhere positive as shown in \cite{MR0240850}. 
\begin{definition}
The L\'evy bridge from $0$ to $0$ of length $1$ is the \cadlag\ process whose law $\p_{0,0}^1$ 
is determined by the local absolute continuity relationship: for every $A\in\F_s$
\begin{esn}
\imf{\p_{0,0}^1}{A}=\esp{\indi{A}\frac{\imf{p_{1-s}}{X_s,0}}{\imf{p_t}{0,0}}}.
\end{esn}
\end{definition}
See \cite{MR628873}, \cite{MR1278079} or \cite{MR2789508} for an interpretation of the above law as that of $X$ conditioned on $X_t=0$. 
Using time reversibility for L\'evy processes, it is easy to see that
the image of $\p_{0,0}^1$ under the time reversal map $\paren{X_{(1-t)-},t\in [0,1]}$ is the bridge of $-X$ from $0$ to $0$ of length $1$ and that $X_1=X_{1-}=0$ under $\p_{0,0}^1$.  

\begin{proposition}
\label{continuityAtMinimumForLevyBridges}
The law $\p_{0,0}^1$ has the EI property. 
Under $\p_{0,0}^1$, the minimum is achieved at a unique place $\rho\in (0,1)$ and $X$ is continuous at $\rho$. 
\end{proposition}
We conclude that Corollary \ref{VervaatCorollary} applies under $\p_{0,0}^1$. 
At this level of generality, this has been proved in \cite{MR3160578}. 
In that work, the 
distribution of the image of $\p_{0,0}^1$ under the Vervaat transformation 
was identified with the (Markovian) bridge associated to the L\'evy process conditioned to stay positive which was constructed there. 
\begin{proof}[Proof of Proposition \ref{continuityAtMinimumForLevyBridges}]
Using the local absolute continuity relationship and the regularity hypothesis \defin{H1} we see that $\underline X<0$ under $\p_{0,0}^1$. 
Let $\delta>0$. 
On $[\delta,1-\delta]$, the laws $\p_{0,0}^1$ and $\p$ are equivalent. 
Since the minimum of $X$ on $[\delta,1-\delta]$ is achieved at a unique place and continuously (because of regularity) under $\p$, the same holds under $\p_{0,0}^1$. 
We now let $\delta\to 0$ and use the fact that $\underline X<0$  under $\p_{0,1}^1$ to conclude. 
\end{proof}

\section{Conditioning a brownian bridge on its minimum}
In Section \ref{motivationSection} we considered a limiting case of Theorem \ref{mainTheorem} by conditioning the minimum of the brownian bridge to equal zero rather than to be close to zero when $X_1=0$. 
In this section, we will show that the limiting procedure is also valid when $X_1>0$ and for any value of the minimum. 
This will enable us to establish, in particular,  a pathwise construction of the Brownian meander. 
\begin{theorem}
\label{generalizedVervaatNonZeroEndingPointTheorem}
Let $\p_x$ be the law of the Brownian bridge from $0$ to $x$ of length $1$. 
Consider the reflected process $R=X-J$ where\begin{esn}
J_t=\inf_{s\in [t,1]} X_s\vee \bra{\underline X_t+X_1}.
\end{esn}Then $R$ admits a bicontinuous family of local times $\paren{L^y_t,t\in [0,1],y\geq 0 }$. 
Let $U$ be a uniform random variable independent of $X$ and define\begin{esn}
\nu=\inf\set{t\geq 0: L^y_t=UL^y_1}.
\end{esn}
Let $\p^{y,x}$ be the law of $\imf{\theta_\nu}{ X}$ conditionally on $L^y_1>0$ . 
Then $\p^{y,x}$ is a version of the law of $X$ given $\underline X=y$ under $\p_x$ which is weakly continuous as a function of $y$. 


Conversely, if $x=0$ and $Y$ has law $\p^{y,0}$, $U$ is a uniform random variable and independent of $Y$, and 
 $A=\overline X-\underline X$ is the amplitude of the path $X$, then $\imf{\theta_U}{ Y} $ has the law of $X$ conditionally on $A\geq -y$. 
\end{theorem}
The process $R$ is introduced in the preceding theorem for a very simple reason: when $X_1\geq 0$, it is equal to $-\underline X\circ \theta_t$. See Figure \ref{reflectedProcessFigure} for an illustration of its definition. 
\begin{figure}
\includegraphics[width=.85\textwidth]{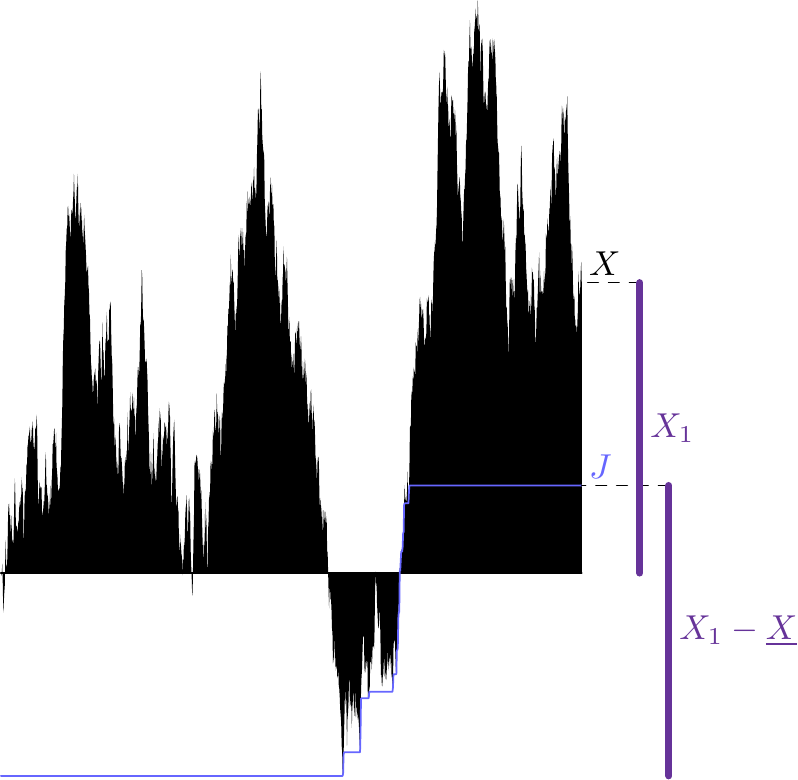}
\includegraphics[width=.85\textwidth]{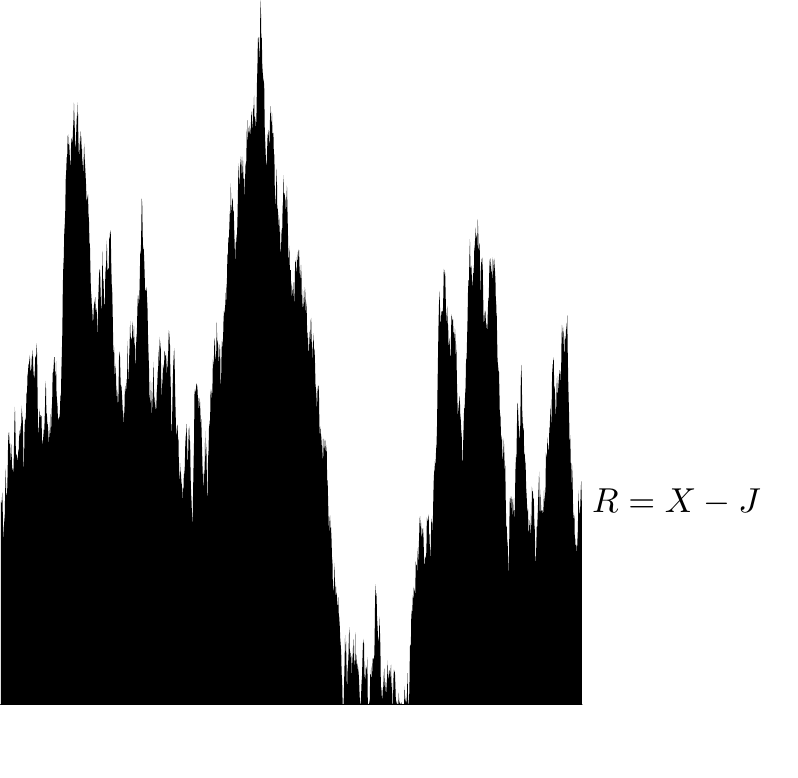}
\caption{Illustration of the reflected process $R$,  given by $-\underline X\circ \theta_t$ when $X_1\geq 0$.}
\label{reflectedProcessFigure}
\end{figure}
\begin{proof}[Proof of Theorem \ref{generalizedVervaatNonZeroEndingPointTheorem}]
To construct the local times, we first divide the trajectory of $X$ in three parts. Let $\rho$ be the unique instant at which the minimum is achieved and let $\underline X$ be the minimum. Using Denisov's decomposition of the Brownian bridge of \cite{denisov:821}, we can see that conditionally on $\rho=t$ and $\underline X=y$, the processes $X^{\leftarrow}=\paren{X_{t-s}-y,s\leq t}$ and $X^{\rightarrow}=\paren{X_{t+s}-y,s\leq 1-t}$ are three-dimensional Bessel bridges starting at $0$, of lengths $t$ and $1-t$, and ending at $y$ and $y+x$ (see also Theorem 3 in \cite{MR3160578}, where the preceding result is stated for $x=0$ for more general L\'evy processes). Next, the trajectory of $X^{\rightarrow}$ will be further decomposed at\begin{esn}
L_x=\sup\set{r\leq 1-t: X^{\rightarrow}_r\leq x}. 
\end{esn}The backward strong Markov property (Theorem 2 in \cite{MR2789508}) tells us that, conditionally on $L_x=s$, the processes $X^{\rightarrow,1}=\paren{X^{\rightarrow}_r,r\leq s}$ is a three-dimensional Bessel bridge from $0$ to $x$ of length $s$. Finally, the process $X^{\rightarrow,2}$ given by $X^{\rightarrow,2}_r=X^{\rightarrow}_{s+r}-x$ for $r\leq 1-t-s$ is a three-dimensional Bessel bridge from $0$ to $y$  of length $1-t-s$. Now, note that under the law of the three-dimensional Bessel process, one can construct a bicontinuous family of local times given as occupation densities. That is, if $\p_0^3$ is the law of the three-dimensional Bessel processes, there exists a bicontinuous process $\paren{L^y_t,t,y\geq 0}$ such that:
\begin{esn}
L_t^y= \lim_{\eps\to 0}\frac{1}{\eps}\int_0^t \indi{\abs{X_s-y}\leq \eps}\, ds
\end{esn}for any $t,y$ almost surely. By Pitman's path transformation between $\p_0^3$ and the reflected Brownian motion found in \cite{MR0375485}, note that if $X_{\rightarrow t}=\inf_{s\geq t}X_s$ is the future infimum process of $X$, then $X-X_{\rightarrow} $ is a reflected Brownian motion for which one can also construct a bicontinuous family of local times. Therefore, the following limits exist and are continuous in $t$ and $y$:
\begin{esn}
L^{r,y}_t=\lim_{\eps\to 0}\int_0^t \indi{\abs{X_r-X_{\rightarrow r}}\leq \eps}\, dr. 
\end{esn}Since the laws of the Bessel bridges are locally absolutely continuous with respect to the law of Bessel processes, we see that the following limits exist and are continuous as functions of $t$ and $z$ under $\p_x$:
\begin{align*}
&\imf{L^{z}_r}{R}
\\&=\lim_{\eps\to 0}\frac{1}{\eps}\int_{[0,r]}\indi{\abs{R_u-z}\leq \eps}\, du
\\&=\lim_{\eps\to 0}\frac{1}{\eps}\bra{\int_{[0,r]} \indi{\abs{X_u-y-z}\leq \eps }\indi{u\in [0,t]\cup[t+s,1]}\, du+\int_{[0,r]} \indi{\abs{X_u-X_{\rightarrow,u}-y-z}\leq \eps }\indi{u\in [t,t+s]}\, du}. 
\end{align*}
(The bridge laws are not absolutely continuous with respect to the original law near the endpoint, but one can then argue by time-reversal.) 

Note that $R$ is cyclically exchangeable, so that the set $L^y_1>0$ is invariant under $\theta_t$ for any $t\in [0,1]$. 
Hence, $X$ conditioned on $L^y_1>0$ is cyclically exchangeable.  
Hence, by conditioning, we can assume that $L^y_1>0$. 

Define $I=(y-\eps,y+\eps)$ and
let\begin{esn}
\eta^I=\inf\set{t\geq 0: A^I_t=U A^I_1}.
\end{esn}Note that the process $L^y$ is strictly increasing at $\nu$. 
Indeed, this happens because $U$ is independent of $L^y$ and therefore is different, almost surely, from any of the values achieved by $L^y/L^y_1$ on any of its denumerable intervals of constancy. 
(The fact that $L^y_1>0$ is used implicitly here.) 
Since $A^I$ converges to $L^y$, it then follows that $\eta^I$ converges to $\nu$. 
Using the fact that $X$ is continuous, it follows that $\theta_{\nu^I}X\to \theta_\nu X$.  
Since  $L^y_1>0$ on, then $\lim_{\eps\to 0}\imf{\p_x}{A^I_1>0}=1$. 
However, by Theorem \ref{mainTheorem}, conditionally on $A^I_1>0$, $\theta_{\eta^I}X$ has the law of $X$ conditioned on $\underline X\in I$. 
Hence, the latter conditional law converges, as $\eps\to 0$ to the law of $\theta_\nu X$. 
A similar argument applies to show that $\nu$ is continuous as a function of $y$ and hence that the law $\p^{y,x}$ is weakly continuous as a function of $y$. 
But now, it is a simple exercise to show that $\paren{\p^{y,x},y\geq 0}$ disintegrates $\p_x$ with respect to $\underline X$. 

Finally, suppose that $x=0$. 
Since $\eta^I$ is independent of $X$, then $\nu$ is independent of $X$ also. 
Hence, the law of $\theta_U$ under $\p^{y,x}$ equals $\p_x$ conditioned on $L_1^y>0$. 
However, note that $L^1_y>0$ implies that $R$ (which equals $X-\underline X$ when $x=0$) reaches level $y$. 
Conversely, if $R$ reaches level $y$, then the local time at $y$ must be positive. 
Hence the sets $\set{L^y_1>0}$ and $\set{A\geq y}$ coincide. 
\end{proof}
One could think of a more general result along the lines of Theorem \ref{generalizedVervaatNonZeroEndingPointTheorem} for L\'evy processes. 
As our proof shows, it would involve technicalities involving local times of discontinuous processes. 
We leave this direction of research open. 

As a corollary (up to a time-reversal), we obtain the path transformation stated as Theorem 7 in \cite{MR2042754}. 
\begin{corollary}
Let $\p$ be the law of a Brownian bridge from $0$ to $0$ of length $1$ and let $U$ be uniform and independent of $X$. 
Let $\nu=\inf\set{t\geq 0: X_t> U \bra{x+\underline X}}$. 
Then $\theta_\nu X$ has the same law as the three-dimensional Bessel bridge from $0$ to $x$ of length $1$. 
\end{corollary}
\begin{proof}
We need only to note that, under the law of the three-dimensional Bessel process, the local time of $X-X_{\rightarrow}$ equals $X_{\rightarrow}$ (which can be thought of as a consequence of Pitman's construction of the three-dimensional Bessel process). Then, the local time at zero of $R$ equals $J+\underline X$, its final value is $x+\underline X$,  and then $\nu=\inf\set{t\geq 0: L^0_{t}>U L^0_1}$. 
\end{proof}
By integrating with respect to $x$ in the preceding corollary, we obtain a path construction of the Brownian meander in terms of Brownian motion. 
Indeed, consider first a Brownian motion $B$ and define $X=B\imf{\sgn}{B_1}$. 
Then $X$ has the law of $B$ conditionally on $B_1>0$ and it is cyclically exchangeable. 
Applying Theorem \ref{generalizedVervaatNonZeroEndingPointTheorem} to $X$, we deduce that if
 $\nu=\inf\set{t\geq 0: X_t> U \bra{X_1+\underline X}}$ then $X\circ\theta_\nu$ has the law of the weak limit as $\eps\to 0$ of $B$ conditioned on $\inf_{t\leq 1} B_t\geq -\eps$, a process which is known as the Brownian meander. 

Setting $x=0$ in Theorem \ref{generalizedVervaatNonZeroEndingPointTheorem} gives us a novel path transformation to condition a Brownian bridge on achieving a minimum equal to $y$. 
In this case, we consider the local time process. 
This generalizes the Vervaat transformation, to which it reduces when $y=0$. 
\begin{corollary}
\label{endingPointEqualZeroMinimumNotZeroCorollary}
Let $\p$ be the law of the Brownian bridge from $0$ to $0$ of length $1$, let $\paren{L^y_t,y\in\re,t\in [0,1]}$ be its continuous family of local times and let $U$ be uniform and independent of $X$. 
For $y\leq 0$, let\begin{esn}
\eta_y=\inf\set{t\geq 0: L^{\underline X-y}_t>UL^{\underline X-y}_1}. 
\end{esn}Then the laws of $X\circ \theta_{\eta_y}$ provide a weakly continuous disintegration of $\p$ given $\underline X=y$. 
\end{corollary}
The only difference with Theorem \ref{generalizedVervaatNonZeroEndingPointTheorem} is that the local times are defined directly in terms of the Brownian bridge since the reflected process $R$ equals $X-\underline X$ when the ending point is zero. 
The equality between both notions follows from bicontinuity and the fact that local times were constructed as limits of occupation times. 
Also, note that since the minimum is achieved in a unique place $\rho\in (0,1)$, then $L^{\underline X}_1=0$. 
Hence $\eta_y\to\rho$ as $y\to 0$ and the preceding path transformation converges to the Vervaat transformation. 

Theorem \ref{generalizedVervaatNonZeroEndingPointTheorem} may be expressed in terms of the non
conditioned process, that is, instead of considering the Bessel
bridge, one may state the above transformation for the three
dimensional Bessel process itself. More precisely, since path
by path
\begin{esn}
\theta_u(X_t-xt\,,0\leq t\leq1)=(\imf{\theta_u f}{X}_t-xt\,,0\leq t\leq1)\,,
\end{esn}
and since the Brownian bridge $b$ from $0$ to $x$  can be represented as 
\begin{equation}\label{rep1}
b_t=X_t-t(X_1-x),\,0\leq t\leq1)\,,
\end{equation}under the law of Brownian motion, 
then the process $(b_t-xt\,,0\leq1)$ is a Brownian bridge from 0 to 0
and then so is $(\theta_u(X)_t-xt\,,0\leq t\leq1)$ under the law of the three dimensional Bessel bridge from $0$ to $x$ of length $1$. 
In particular,
the law of the latter process does not depend on $x$ and we can state:
\begin{corollary}\label{bes3Corollary} 
Under the law of the three-dimensional Bessel process on $[0,1]$, if $U$ is uniform and independent of $X$, 
then $\paren{\imf{\theta_U}{X}_t-tX_1,\,0\leq t\le1}$ is a Brownian
bridge (from $0$ to $0$ of length $1$) which is independent of $X_1$.
\end{corollary}

Let us end by noting the following consequence of Theorem \ref{generalizedVervaatNonZeroEndingPointTheorem}: if $\p$ is the law of the Brownian bridge from $0$ to $0$, then  the law of $\overline X-y$ given $\underline X=y$ equals the law of $A$ given $A\geq y$. When $y=0$, we conclude that the law of the maximum of a normalized Brownian excursion equals the law of the range of a Brownian bridge. This equality was first proved in \cite{MR0467948} and \cite{MR0402955}. Providing a probabilistic explanation was the original motivation of Vervaat when proposing the path transformation $\theta_\rho$ in \cite{MR515820}. 
\bibliography{GenBib}

\providecommand{\bysame}{\leavevmode\hbox to3em{\hrulefill}\thinspace}
\providecommand{\MR}{\relax\ifhmode\unskip\space\fi MR }
\providecommand{\MRhref}[2]{%
  \href{http://www.ams.org/mathscinet-getitem?mr=#1}{#2}
}
\providecommand{\href}[2]{#2}
\begin{thebibliography}{CUB11}

\bibitem[Ald97]{multiplicativeCoalescentAldous}
David Aldous, \emph{Brownian excursions, critical random graphs and the
  multiplicative coalescent}, Ann. Probab. \textbf{25} (1997), no.~2, 812--854.
  \MR{1434128}

\bibitem[BCP03]{MR2042754}
Jean Bertoin, Lo{\"{\i}}c Chaumont, and Jim Pitman, \emph{Path transformations
  of first passage bridges}, Electron. Comm. Probab. \textbf{8} (2003),
  155--166 (electronic). \MR{2042754}

\bibitem[Ber96]{MR1406564}
Jean Bertoin, \emph{L{\'e}vy processes}, Cambridge Tracts in Mathematics, vol.
  121, Cambridge University Press, Cambridge, 1996. \MR{1406564}

\bibitem[Ber97]{MR1465812}
\bysame, \emph{Regularity of the half-line for {L}{\'e}vy processes}, Bull.
  Sci. Math. \textbf{121} (1997), no.~5, 345--354. \MR{1465812}

\bibitem[Blu83]{MR704566}
R.~M. Blumenthal, \emph{Weak convergence to {B}rownian excursion}, Ann. Probab.
  \textbf{11} (1983), no.~3, 798--800. \MR{704566}

\bibitem[Chu76]{MR0467948}
Kai~Lai Chung, \emph{Excursions in {B}rownian motion}, Ark. Mat. \textbf{14}
  (1976), no.~2, 155--177. \MR{0467948 (57 \#7791)}

\bibitem[CUB11]{MR2789508}
Lo{\"{\i}}c Chaumont and Ger{\'o}nimo Uribe~Bravo, \emph{Markovian bridges:
  weak continuity and pathwise constructions}, Ann. Probab. \textbf{39} (2011),
  no.~2, 609--647. \MR{2789508}

\bibitem[Den84]{denisov:821}
I.~V. Denisov, \emph{A random walk and a {W}iener process near a maximum},
  Theory of Probability and its Applications \textbf{28} (1984), no.~4,
  821--824.


\bibitem[DIM77]{MR0436353}
Richard~T. Durrett, Donald~L. Iglehart, and Douglas~R. Miller, \emph{Weak
  convergence to {B}rownian meander and {B}rownian excursion}, Ann. Probability
  \textbf{5} (1977), no.~1, 117--129. \MR{0436353}

\bibitem[Don07]{MR2320889}
Ronald~A. Doney, \emph{Fluctuation theory for {L}{\'e}vy processes}, Lecture
  Notes in Mathematics, vol. 1897, Springer, Berlin, 2007. \MR{2320889}

\bibitem[FPY93]{MR1278079}
Pat Fitzsimmons, Jim Pitman, and Marc Yor, \emph{Markovian bridges:
  construction, {P}alm interpretation, and splicing}, Seminar on {S}tochastic
  {P}rocesses, 1992 
  Birkh{\"a}user Boston, 1993, pp.~101--134. \MR{1278079}

\bibitem[GY93]{MAFI:MAFI349}
H{\'e}lyette Geman and Marc Yor, \emph{Bessel processes, asian options, and
  perpetuities}, Mathematical Finance \textbf{3} (1993), no.~4, 349--375.

\bibitem[It{\^o}72]{MR0402949}
Kiyosi It{\^o}, \emph{Poisson point processes attached to {M}arkov processes},
  Proceedings of the {S}ixth {B}erkeley {S}ymposium on {M}athematical
  {S}tatistics and {P}robability ({U}niv. {C}alifornia, {B}erkeley, {C}alif.,
  1970/1971), {V}ol. {III}: {P}robability theory, Univ. California Press, 1972,
  pp.~225--239. \MR{0402949}

\bibitem[Kal73]{MR0394842}
Olav Kallenberg, \emph{Canonical representations and convergence criteria for
  processes with interchangeable increments}, Z. Wahrscheinlichkeitstheorie und
  Verw. Gebiete \textbf{27} (1973), 23--36. \MR{0394842}

\bibitem[Kal81]{MR628873}
\bysame, \emph{Splitting at backward times in regenerative sets}, Ann. Probab.
  \textbf{9} (1981), no.~5, 781--799. \MR{628873}

\bibitem[Kal74]{MR0402901}
\bysame, \emph{Path properties of processes with independent and
  interchangeable increments}, Z. Wahrscheinlichkeitstheorie und Verw. Gebiete
  \textbf{28} (1973/74), 257--271. \MR{0402901}

\bibitem[Ken76]{MR0402955}
Douglas~P. Kennedy, \emph{The distribution of the maximum {B}rownian
  excursion}, J. Appl. Probability \textbf{13} (1976), no.~2, 371--376.
  \MR{0402955 (53 \#6769)}

\bibitem[Kyp06]{MR2250061}
Andreas~E. Kyprianou, \emph{Introductory lectures on fluctuations of {L}{\'e}vy
  processes with applications}, Universitext, Springer-Verlag, Berlin, 2006.
  \MR{2250061}

\bibitem[L{\'e}v39]{MR0000919}
Paul L{\'e}vy, \emph{Sur certains processus stochastiques homog\`enes},
  Compositio Math. \textbf{7} (1939), 283--339. \MR{0000919}

\bibitem[LG10]{MR2603061}
Jean-Fran{\c{c}}ois Le~Gall, \emph{It\^o's excursion theory and random trees},
  Stochastic Process. Appl. \textbf{120} (2010), no.~5, 721--749. \MR{2603061}

\bibitem[Mil77]{MR0433606}
P.~W. Millar, \emph{Zero-one laws and the minimum of a {M}arkov process},
  Trans. Amer. Math. Soc. \textbf{226} (1977), 365--391. \MR{0433606}

\bibitem[Pit75]{MR0375485}
J.~W. Pitman, \emph{One-dimensional {B}rownian motion and the three-dimensional
  {B}essel process}, Advances in Appl. Probability \textbf{7} (1975), no.~3,
  511--526. \MR{0375485 (51 \#11677)}

\bibitem[PY07]{MR2295611}
J.~Pitman and M.~Yor, \emph{It\^o's excursion theory and its applications},
  Jpn. J. Math. \textbf{2} (2007), no.~1, 83--96. \MR{2295611}

\bibitem[Rog68]{MR0242261}
B.~A. Rogozin, \emph{The local behavior of processes with independent
  increments}, Teor. Verojatnost. i Primenen. \textbf{13} (1968), 507--512.
  \MR{0242261}

\bibitem[Sha69]{MR0240850}
Michael Sharpe, \emph{Zeroes of infinitely divisible densities}, Ann. Math.
  Statist. \textbf{40} (1969), 1503--1505. \MR{0240850}

\bibitem[{\v{S}}ta65]{MR0183022}
E.~S. {\v{S}}tatland, \emph{On local properties of processes with independent
  increments}, Teor. Verojatnost. i Primenen. \textbf{10} (1965), 344--350.
  \MR{0183022}

\bibitem[UB14]{MR3160578}
Ger{\'o}nimo Uribe~Bravo, \emph{Bridges of {L}\'evy processes conditioned to
  stay positive}, Bernoulli \textbf{20} (2014), no.~1, 190--206. \MR{3160578}

\bibitem[Ver79]{MR515820}
Wim Vervaat, \emph{A relation between {B}rownian bridge and {B}rownian
  excursion}, Ann. Probab. \textbf{7} (1979), no.~1, 143--149. \MR{515820}

\bibitem[Wat10]{MR2603058}
Shinzo Watanabe, \emph{It\^o's theory of excursion point processes and its
  developments}, Stochastic Process. Appl. \textbf{120} (2010), no.~5,
  653--677. \MR{2603058}

\bibitem[Wer10]{MR2603062}
Wendelin Werner, \emph{Poisson point processes, excursions and stable processes
  in two-dimensional structures}, Stochastic Process. Appl. \textbf{120}
  (2010), no.~5, 750--766. \MR{2603062}

\bibitem[YY13]{MR3134857}
Ju-Yi Yen and Marc Yor, \emph{Local times and excursion theory for {B}rownian
  motion}, Lecture Notes in Mathematics, vol. 2088, Springer, 2013.
  \MR{3134857}

\end{thebibliography}
\bibliographystyle{amsalpha}
%
\end{document}